\documentclass[11pt]{article}
\usepackage{amsfonts,amsmath,amssymb,amsthm, amscd, subcaption}
\usepackage{float}
\usepackage{graphicx}
\numberwithin{equation}{section}

\newtheorem{theorem}{Theorem}[section]
\newtheorem{prop}[theorem]{Proposition}
\newtheorem{lemma}[theorem]{Lemma}
\newtheorem{cor}[theorem]{Corollary}

\theoremstyle{definition}
\newtheorem{definition}[theorem]{Definition}
\newtheorem{example}[theorem]{Example}
\newtheorem{remark}[theorem]{Remark}
\newtheorem{question}[theorem]{Question}

\def\<{{\langle}}
\def\>{{\rangle}}
\def\G{{\Gamma}}
\def\a{{\alpha}}
\def\b{{\beta}}
\def\g{{\gamma}}
\def\d{{\delta}}
\def\Z{\mathbb Z}

\def\R{\mathbb R}
\def\T{{\mathbb T}}
\def\S{{\mathbb S}}

\def\De{\Delta}
\def\D{\cal D}

\def\a{\alpha}
\def\b{\beta}

\def\si{\sigma}
\def\t{\tau}
\def\k{{\kappa}}

\def\Gr{\mathbb G}
\def\L{\cal L}
\def\La{{\Lambda}}

\def\Rd{{\cal R}_d}
\def\n{{\bf n}}
\def\e{\epsilon}
\def\A{\mathbb A}
\def\o{\overline}
\def\bs{\bigskip}

\def\n{{\bf n}}
\def\s{{\bf s}}
\def\g{\gamma}

\def\ni{\noindent} 

\begin{document}

\title{Graph Complexity and Link Colorings}

\author{Daniel S. Silver 
\and Susan G. Williams\thanks {Both authors are grateful for the support of the Simons Foundation.} }

\maketitle 


\begin{abstract} The \emph{(torsion) complexity} of a finite signed graph is defined to be the order of the torsion subgroup of the abelian group presented by its Laplacian matrix. When $G$ is $d$-periodic (i.e., $G$ has a free $\Z^d$-action by graph automorphisms with finite quotient) the Mahler measure of its Laplacian polynomial is the growth rate of the complexity of  finite quotients of $G$.  Any 1-periodic plane graph $G$ determines a link $\ell \cup C$ with unknotted component $C$. In this case the Laplacian polynomial of $G$ is  related to the Alexander polynomial of the link. Lehmer's question, an open question about the roots of monic integral polynomials, is equivalent to a question about the complexity growth of signed 1-periodic graphs that are not necessarily embedded. \bigskip

\end{abstract}

\section{Motivation} \label{Intro} Consider a locally finite graph without multiple edges or isolated vertices. 
Let $p$ be an odd prime. We can attempt to form a ``$p$-coloring" of the graph, assigning elements (``colors") of $\Z/p$ to the vertices in such a way that the color of any vertex is the average of the colors of the adjacent vertices. Obviously, the graph can be colored trivially or ``monochromatically," that is, with any single color. But nontrivial $p$-colorings might exist as well. It is clear that the set of all $p$-colorings of the graph is a vector space under vertex-wise addition and scalar multiplication.  

Now consider a plane diagram of a knot or link. (As a knot is a link with only one component, henceforth we will use the more general term.) We can try to form a ``$p$-coloring,"  assigning colors to the arcs in such a way that the color of any overcrossing arc is the average of the colors of its two undercrossing arcs (see Section \ref{Links}). Again we can color trivially, but nontrivial $p$-colorings can exist too. The set of the all $p$-colorings of the diagram is a vector space under arc-wise addition and scalar multiplication. 

There is a strong relationship between $p$-colorings of plane graphs and $p$-colorings of link diagrams \cite{STW20, SW20'}. We review the relationship in Section \ref{diagrams}.
By substituting the continuous palette $\T= \R/\Z$, which contains $\Z/p$ as a subgroup, and replacing the plane with a compact orientable surface $S$, we enter the world of algebraic dynamical systems \cite{SW00}. When $S$ is an annulus, we obtain a version of Lehmer's unanswered question about roots of monic integral polynomials, which we formulate in terms of signed graphs. \bs

\ni {\bf Acknowledgements.} It is the authors' pleasure to thank Abhijit Champanerkar, Eriko Hironaka, Matilde Lalin and Chris Smythe for helpful comments and suggestions.

\section{Coloring and complexity of finite graphs} \label{Coloring}

Let $G$ be a graph, not necessarily planar, with vertex set $V(G)=\{v_1, \ldots, v_n\}$ and edge set $E(G)=\{e_1, \ldots, e_m\}$. The graph is allowed to have loops and multiple edges, but no isolated vertices. Each edge $e \in E(G)$ is labeled with a sign $\si_e=\pm1$.  All graphs that we consider are assumed to have signed edges. The graph is \emph{unsigned} if every $\si_e =+1$.

The \emph{signed adjacency matrix} of $G$ is the $n \times n$ matrix 
$A = (a_{i, j})$ such that $a_{i, j}$ is the sum of the signs of all edges between $v_i$ and $v_j$, with loops counted twice. Define the \emph{signed degree matrix} $\d= (\d_{i, j})$ to be the $n \times n$ diagonal matrix with $\d_{i,i}$ equal to the sum of signs of edges incident on $v_i$. Again, loops contribute twice.

An integer matrix presents an abelian group in which columns (resp. rows) correspond to the generators (resp. relations) of the group. Equivalently, the group is the cokernel of the matrix when regarded as a linear transformation of free abelian groups.

\begin{definition} \label{complexity} The \emph{Laplacian matrix} of a finite graph $G$ is $L_G=\d - A$. The abelian group presented by $L_G$ is the \emph{Laplacian group} of $G$, denoted by ${\L}_G$. The \emph{(torsion) complexity} $\k_G$ is the order of the torsion subgroup $T{\L}_G$.
\end{definition}

We note that in computing the integer matrix $L_G$ we may ignore loops, since they contribute equally to $A$ and $\d$.  However, this will not be the case for the extended definition we make in Section \ref{modules}.

When we regard the entries of $L_G$ modulo $p$, the rows represent the relations needed to $p$-color the graph. (Here we are extending the notion of $p$-coloring to signed graphs.) Hence the space of $p$-colorings of $G$ is isomorphic to the null space of $L_G$.  Nontrivial $p$-colorings exist precisely when its dimension is greater than 1. Note that the abelian group of $p$-colorings of $G$ is isomorphic to ${\L}_G \otimes \Z/p$.

Returning to integer coefficients, the nullity of the Laplacian matrix $L_G$ is equal to 1 whenever $G$ is connected and unsigned (see, for example, Lemma 13.1.1 of \cite{GR01}). In this case the Laplacian group ${\L}_G$ decomposes as the direct sum of $\Z$ and the torsion subgroup $T{\L}_G$. The Matrix Tree Theorem (c.f. \cite{Tu84}) implies that $\k_G = |T{\L}_G|$ is equal to the number of spanning trees of $G$.

More generally, we define \emph{tree complexity}  $\t_G$ of a connected graph $G$ by
\begin{equation*} \t_G =\bigg \vert \sum_{T} \prod_{e \in E(T)} \si_e\bigg \vert, \end{equation*}
where the summation is taken over all spanning trees of $G$. If $G$ not connected, then we define $\t_G$ to be the product of the tree complexities of its connected components.  Again by \cite{Tu84}, we have $\t_G = \k_G$ if and only if $\t_G$ is nonzero; for connected $G$ this common value is equal to the absolute value of any $(n-1) \times (n-1)$ principal minor of $L_G$. However,  the following example shows that $\t_G$ can vanish, whereas $\k_G$ is positive by definition.

\begin{example}\label{milnor}  Consider the connected graph $G$ in Figure \ref{milnor1}. Unlabeled edges here and throughout will be assumed to have sign $+1$. The Laplacian matrix $L_G$ is square of size $8$. A routine calculation shows that any principal $7 \times 7$ minor of $L_G$ vanishes, and hence $\t_G=0$.   However, the absolute value of the greatest common divisors of the $6 \times 6$ minors of $L_G$ is 9, and so $\k_G =9$. 

We can go further by computing the Smith Normal Form of $L_G$. We then see that ${\L}_G \cong \Z^2 \oplus \Z/3 \oplus \Z/3$. The vector space of $p$-colorings has dimension 2 for all primes $p \ne 3$, while the space of $3$-colorings of $G$ has dimension 4. 
\end{example} 

\begin{figure}[H]
\begin{center}
\includegraphics[height=2 in]{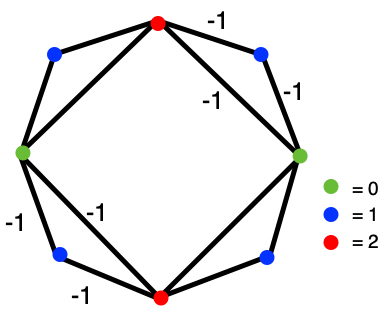}
\caption{Graph $G$ with $\t_G=0$ and $\k_G = 9$}
\label{milnor1}
\end{center}
\end{figure}


\section{Link diagrams and Tait graphs} \label{diagrams} A \emph{link} $\ell$ in $\R^3$ (or equivalently $\S^3$) is a set of smoothly embedded, pairwise disjoint simple closed curves. Any link can be described by a \emph{diagram} $D$, a 4-valent plane graph with a hidden-line device in a neighborhood of each vertex indicating how one strand of the link passes over another. 

We will make use of a few other related terms. An \emph{arc} of $D$ is a maximal connected subset. Following \cite{Ka01}, we refer to the underlying graph of $D$ as a \emph{universe} of $\ell$, denoted here by $|D|$. Finally, a \emph{region} of $D$ is a connected component of $\R^2 \setminus |D|$. 

As usual, isotopic links are regarded as the same. It is well known that two links are isotopic if and only if  a diagram of one can be transformed into a diagram of the other by a finite sequence of local ``Reidemeister moves" as in Figure \ref{reid} (see \cite{BZ03} or \cite{Li97} for details). Consequently, any quantity that is determined by a diagram and is unchanged by Reidemeister moves is a link invariant. 

\begin{figure}[H]
\begin{center}
\includegraphics[height=2 in]{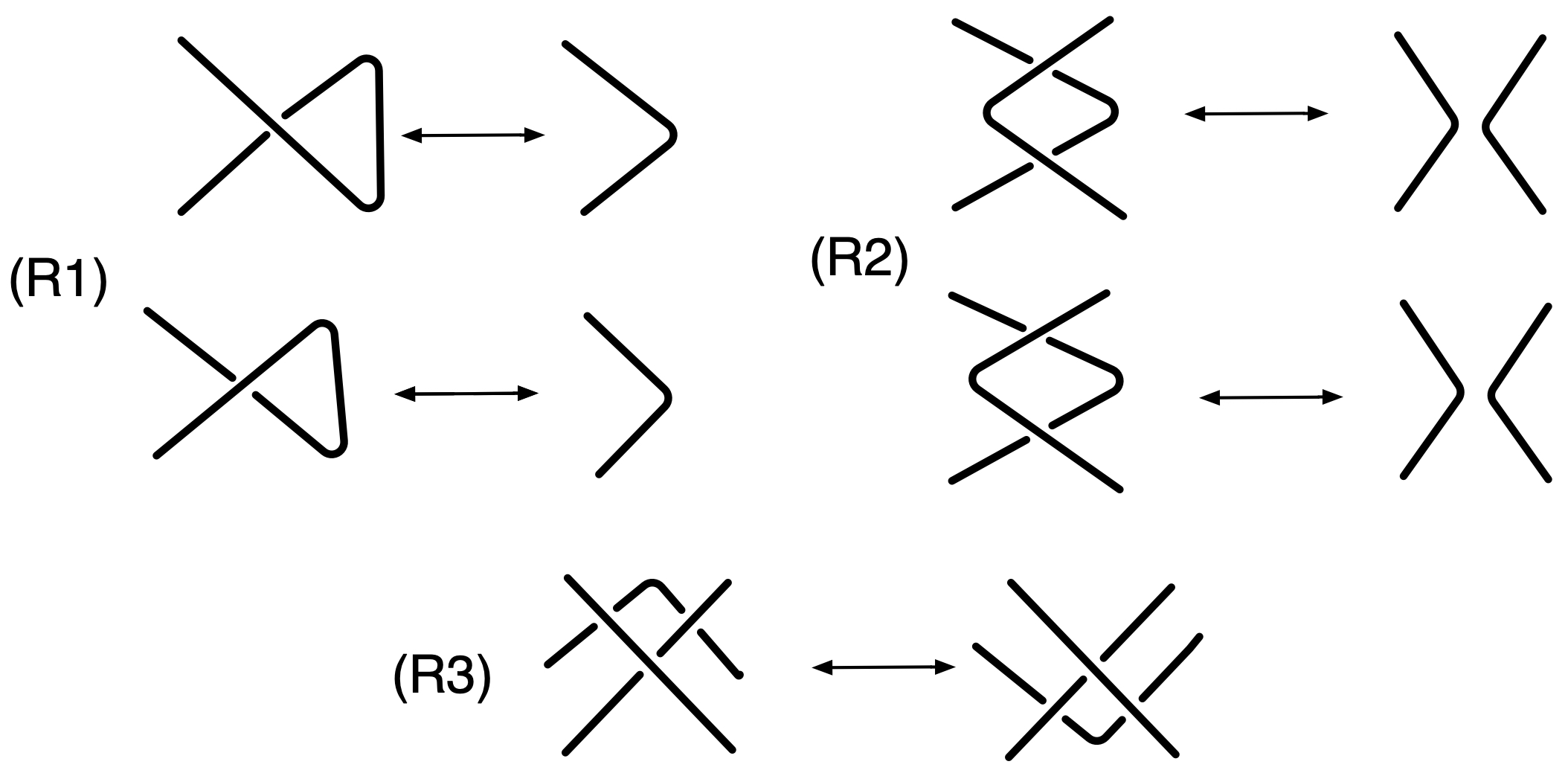}
\caption{Reidemeister moves}
\label{reid}
\end{center}
\end{figure}

We can obtain a plane graph $G$ from a link diagram $D$ by the following familiar procedure.  First we \emph{checkerboard shade} $D$, shading some of the regions so that every edge meets a single shaded region. There are two checkerboard shadings of $D$, but for the sake of definiteness we will choose the one in which the unbounded region is unshaded. We construct a graph $G$ with a vertex in each shaded region of $D$ and an edge through each crossing, joining the vertices of the regions on both sides. The sign of the edge is determined by the type of crossing, as in Figure \ref{medialgraph}. Such a graph is often called a ``Tait graph," in honor of Peter Guthrie Tait, a 19th century pioneer of knot theory.

\begin{figure}[H]
\begin{center}
\includegraphics[height=1.5 in]{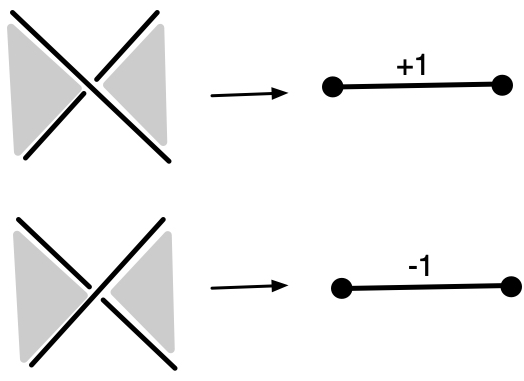}
\caption{Constructing a graph from a link diagram}
\label{medialgraph}
\end{center}
\end{figure}

For some link diagrams it can happen that the Tait graph has isolated vertices. We avoid this by always assuming that $|D|$ is connected, a condition that we can assume without loss of generality by applying Reidemeister moves to $D$. 

For a plane graph $G$ we can reverse the above procedure to obtain a link digram $D$. For this we use the \emph{medial construction}, replacing each edge of $G$ by a pair of arc segments that run parallel to the edge except at the middle, where they cross, as in Figure \ref{medial}.  We join the segments near vertices in the obvious way without creating any additional crossings. 

Figure \ref{milnor2} shows the diagram of a 3-component link $\ell$ obtained from the graph of Example \ref{milnor}. According to \cite{Co80} the link was introduced by J. Milnor. Original interest in the link came from the fact that it was the first non-alternating boundary link discovered with zero Alexander polynomial, a fact that we will not use here. 
We will return to the link in the next section. 

\begin{figure}[H]
\begin{center}
\includegraphics[height=1 in]{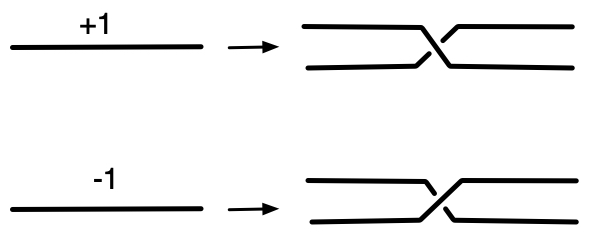}
\caption{Constructing a link diagram from a plane graph}
\label{medial}
\end{center}
\end{figure}
\section{Coloring link diagrams}\label{Links} 

Assume that $D$ is a diagram of a link $\ell$. Let $p$ be a prime. A \emph{(Fox) $p$-coloring} of $D$ is an assignment of elements of $\Z/p$, called \emph{colors}, to the arcs of $D$ such that twice the color of any overcrossing arc is equal to the sum of the colors of its undercrossing arcs, as in Figure \ref{fox}. A $p$-coloring is \emph{trivial} if all assigned colors are the same. 

\begin{figure}[H]
\begin{center}
\includegraphics[height=1 in]{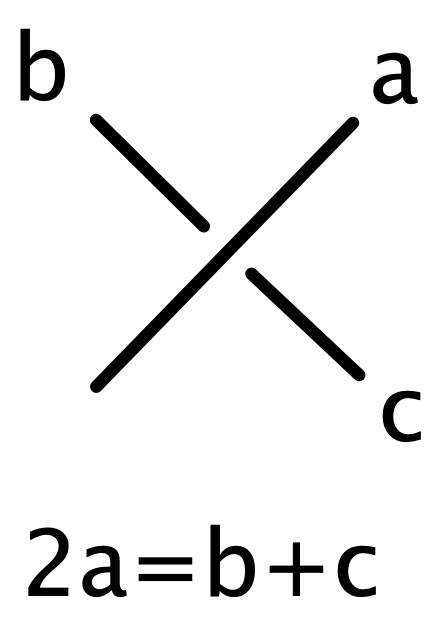}
\caption{Fox coloring relation}
\label{fox}
\end{center}
\end{figure}

We denote the vector space of $p$-colorings of $D$ by ${\rm Col}_p(D)$ and refer to it as the \emph{$p$-coloring space} of the diagram. It is a vector space over $\Z/p$, with addition and scalar multiplication performed arc-wise. A routine exercise shows that diagrams differing by a Reidemeister move have isomorphic $p$-coloring spaces. Hence ${\rm Col}_p(D)$ is an invariant of the link $\ell$, and so we write ${\rm Col}_p(\ell)$ and speak of the vector space as the $p$-coloring space of the link. (For more information about this popular construction of knot theory, there are many excellent expositions such as \cite{Li93}.)

Fox envisioned $p$-colorings as homomorphisms from the link group 
$\pi_\ell = \pi_1(\R^3 \setminus \ell)$ onto the dihedral group
$$D_{2p} = \< \t, \a \mid \t^2=1, \a^p=1, \a \t = \t \a^{-1} \>.$$ 
He described the correspondence using the Wirtinger presentation of $\pi_\ell$, in which generators (resp. relations) are identified with arcs (resp. crossings) of a diagram $D$. Given a $p$-coloring of $D$ we obtain a homomorphism from $\pi_\ell$ to $D_{2p}$ by sending a Wirtinger generator of an arc colored by $k \in \Z/p$ to the element $\t \a^k$ of $D_{2p}$.

If instead of the Wirtinger presentation, we use the Dehn presentation of $\pi_\ell$, a presentation in which
generators (resp. relations) are identified with bounded regions (resp. crossings), then a $p$-coloring becomes an assignment of colors to the bounded regions of $D$. We assign $0$ to the unbounded region. The condition at each crossing  that corresponds to the Fox $p$-coloring condition appears in Figure \ref{dehncol}. Details can be found in \cite{SW20}. We will call such an assignment of colors to the regions a \emph{Dehn $p$-coloring} of the diagram. The collection of all Dehn $p$-colorings of a diagram is vector space under region-wise addition and scalar multiplication, isomorphic to the space of Fox $p$-colorings. 

\begin{figure}[H]
\begin{center}
\includegraphics[height=1 in]{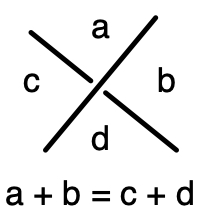}
\caption{Dehn $p$-coloring condition}
\label{dehncol}
\end{center}
\end{figure}

Assume that $D$ is a link diagram arising from a plane graph $G$ by the medial construction. Any $p$-coloring of $G$ determines Dehn and Fox $p$-colorings of $D$. To see this, first assign the colors of the vertices of $G$ to the associated shaded regions of the link diagram. Then use the Dehn coloring relations  to determine uniquely the colors of the unshaded regions. This is possible since the unbounded region is already labeled (with $0$). Uniqueness follows from the observation that if we determine the color of some unshaded region and then follow a simple closed path around a vertex, determining the colors of successive unshaded regions along the way, then when we return to the initial unshaded region, the Laplace relation forces us to arrive at the same color with which we began. (More details can be found in \cite{SW20}.) Finally, assign to each arc of $D$ the sum of the colors of the regions on both sides. It is easy to verify that we obtain in this way is well defined Fox $p$-coloring of the diagram. 

Conversely, any Fox $p$-coloring of a link diagram determines a $p$-coloring of the associated Tait graph. The Dehn color of any region is the sum of the colors of the arcs that we cross traveling along any path to the unbounded region. The graph coloring is simply a restriction to the shaded regions.

The two processes above are inverses of each other. Hence the vector spaces of $p$-colorings of $G$ and $D$ are isomorphic.

As an example, consider the diagram of Milnor's boundary link as in appears in Figure \ref{milnor2}. The Fox 3-coloring displayed corresponds to the 3-coloring of the graph G in Figure \ref{milnor1}.

\begin{figure}[H]
\begin{center}
\includegraphics[height=2 in]{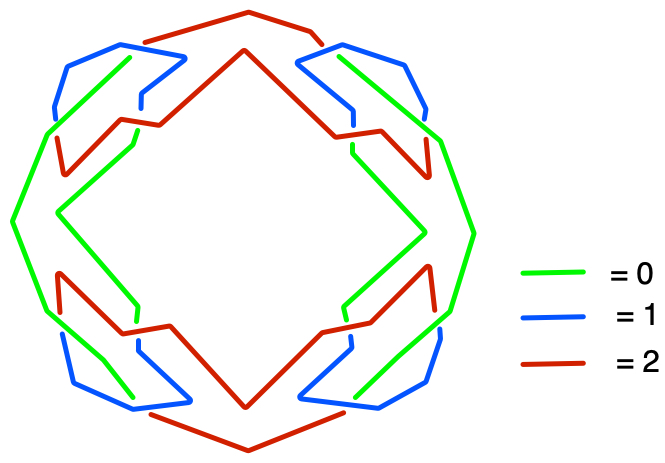}
\caption{$3$-colored diagram of Milnor's boundary link}
\label{milnor2}
\end{center}
\end{figure}

Every finite cyclic group is contained in the compact abelian ``circle group" $\T = \R/\Z$ as a subgroup. If we replace $\Z/p$ by $\T$, then the coloring vector spaces for graphs and link diagrams become compact abelian groups. (This extension of Fox $p$-coloring was introduced in \cite{SW00'}.) The Laplace group ${\L}_G$ tensored with $\T$ consists of $|T{\L}_G|$ tori each of dimension equal to the nullity of $L_G$. A similar description of the circle-coloring group ${\rm Col}_\T(\ell)$ applies. In the following sections we will explore this idea for infinite graphs and links with symmetries. The result is a rich structure that brings algebraic dynamics into our story.

\section{Periodic graphs and Laplacian modules} \label{modules} A (signed) graph $G$ is \emph{$d$-periodic} if it admits a cofinite free $\Z^d$-action by automorphisms that preserves the signs of edges. By \emph{cofinite} we mean the quotient  graph $\o G$ is finite, while an action is \emph{free} if the stabilizer of any edge or vertex is trivial. Such a graph $G$ is locally finite, and finite if and only if $d=0$.

We regard $\Z^d, d \ge 1,$ as the multiplicative abelian group freely generated by $x_1, \ldots, x_d$. We denote the Laurent polynomial ring $\Z[\Z^d] = \Z[x_1^{\pm 1}, \ldots, x_d^{\pm 1}]$ by $\Rd$. As an abelian group $\Rd$ is generated freely by monomials 
$x^\n = x_1^{n_1} \ldots x_d^{n_d}$, where $\n = (n_1, \cdots, n_d) \in \Z^d$. We represent $(0, \ldots, 0)$ by ${\bf 0}$. For notational convenience, when $d=1$, we replace $x_1$ by $x$. 

The vertex set $V(G)$ and the edge set $E(G)$ consist of finitely many vertex orbits $\{v_{1, \n}\mid \n \in \Z^d\}, \ldots, \{v_{n, \n}\mid \n \in \Z^d\}$ and signed edge orbits $\{e_{1, \n}\mid \n \in \Z^d\}, \ldots, \{e_{m, \n} \mid \n \in \Z^d\}$, respectively. The $\Z^d$-action is determined by 
\begin{equation} x^{\n'} \cdot v_{i, \n} = v_{i, \n+ \n'}, \quad \quad   x^{\n'} \cdot e_{j, \n} = e_{j, \n+ \n'},\end{equation}
where $1\le i \le n,\ 1 \le j \le m$ and $\n, \n' \in \Z^d$. (When $G$ is embedded in some Euclidean space with $\Z^d$ acting by translation, it is usually called a \emph{lattice graph}. Such graphs arise frequently in physics, for example in studying crystal structures.)

When $d>1$ we can think of $G$ as covering a finite graph $\o G$ in the $d$-torus $\T^d = \R^d/\Z^d$.  When $d=1$, $G$ covers a finite graph $\o G$ in the annulus $\A=  I \times \S^1$. In either case the cardinality  $|V(\o G)|$ is equal to the number $n$ of vertex orbits of $G$, while $|E(\o G)|$ is the number $m$ of edge orbits. The projection map is given by $v_{i, \n} \mapsto v_i$ and $e_{j, \n} \mapsto e_j$.


The \emph{Laplacian matrix} of a $d$-periodic graph $G$ is defined to be the $n \times n$-matrix $L_G = \d - A$, where now $A = (a_{i,j})$ is the \emph{adjacency $\Rd$-matrix} with each entry $a_{i, j}$  equal to the sum of  monomials $\si_e x^\n$ for each edge $e \in E( \o G)$ between $v_{i, \bf{0}}$ and $v_{j, \n}$. Loops contribute two diagonal terms that are inverse monomials. The signed degree matrix $\d=(\d_{i,j})$ is defined as in Section \ref{Coloring}.

The matrix $L_G$ presents a finitely generated $\Rd$-module, the \emph{Laplacian module} of $G$, denoted by ${\L}_G$. 
The \emph{Laplacian (determinant) polynomial} $\De_G$ is the determinant of $L_G$. When $d=0$, $\Rd=\Z$ and these definitions reduce to the ones in Section \ref{Coloring}. Examples appear below; additional examples can be found in \cite{LSW14, SW20}.
The reader should be aware that in graph theory literature the term ``Laplacian polynomial" is often used for the characteristic polynomial of the integral Laplacian matrix. 

\section{Coloring periodic graphs} \label{col per} Let $G$ be a $d$-periodic graph. The collection of all $\T$-colorings of $G$ is the 
Pontryagin dual group $\widehat {\L}_G = {\rm Hom}({\L}_G, \T)$. Elements are functions $f: V(G) \to \T$ that assign to each 
vertex $v_{i, \n} \in V(G)$ a color $f(v_{i, \n})  \in \T$ such that the Laplacian condition (corresponding to the $i$th row of $L_G$) is satisfied: 
\begin{equation}\label{color} \d_{i,i} f(v_{i, \n}) = \sum_e \si_e f(v_{i', \n'}),\end{equation} 
where $\d_{i,i}$ is the signed degree of $v_{i, \n}$, and the summation is taken over all edges $e$ that connect $v_{i, \n}$ with some $v_{i', \n'} \in V(G)$, loops contributing twice.

We regard ${\L}_G$ with the discrete topology. Endowed with the compact-open topology,
 $\widehat{\L}_G$ is a compact space (see, for example, Section 2 of \cite{LSW90}).  It admits a $\Z^d$-action by automorphisms. Such an action is a homomorphism $\s: \n \mapsto \s_{\n}$ from $\Z^d$ to the automorphism group of 
$\widehat {\L}_G$. 

We denote $\widehat {\L}_G$ with its $\Z^d$-action by ${\rm Col}_{\T, \Z^d}(G)$. It is an example of a dynamical system known as a \emph{$\Z^d$-shift}. By the Pontryagin Duality Theorem we recover the Laplace module by taking the dual of ${\rm Col}_{\T, \Z^d}(G)$.
We will say more about the dynamical properties of $\widehat {\L}_G$ in Section \ref{knots}.

In the case of a finite plane graph and its associated link diagram ($d=0$), their isomorphic groups of $\T$-colorings are dual respectively to the Laplacian group of the graph and the \emph{abelian core group} of the link. The latter group is generated by the arcs of the diagram with relations given by the Fox coloring condition of Figure \ref{fox}; it is well known to be isomorphic to the direct sum of the first homology group of the 2-fold branched cover of the link and an infinite cyclic group.  (For more about such dynamical systems see \cite{Sc95, LSW90} or \cite{SW00'}. Information about the core group can be found in \cite{SW06}.)
\section{Computing the Laplacian polynomial}

A \emph{cycle-rooted spanning forest} (CRSF) of $\o G$  is a subgraph of $\o G$  containing all of $\o V$ such that each connected component has exactly as many vertices as edges and therefore has a unique cycle. The \emph{connection}  $\phi$ of an oriented cycle is its homology class in $H_1(\T^d; \Z) \cong \Z^d$.
See \cite{Ke11} for details. 

The following is a consequence of the main theorem of \cite{Fo93}. It is made explicit in Theorem 5.2 of \cite{Ke11}.

\begin{theorem}\label{poly} \cite{Ke11} Let $G$ be a $d$-periodic graph. Its Laplacian polynomial has the form \begin{equation}\label{polys} \De_G = \sum_{\o F}\  \prod_{\o e \in E(\o F)} \si_{\o e} \prod_{\rm Cycles\ of\ \o F} (2-\phi -\phi^{-1}),\end{equation} where  the sum is over all cycle-rooted spanning forests $\o F$ of $\o G$, and $\phi, \phi^{-1}$ are the connections of the two orientations of the cycle. \end{theorem}

A $d$-periodic graph need not be connected. In fact, it can have countably many connected components. Nevertheless, the number of  $\Z^d$-orbits of components, henceforth called \emph{component orbits}, is necessarily finite. 

\begin{prop} \label{components}  If $G$ is a $d$-periodic graph with component orbits $G_1, \ldots, G_t$, then \break $\De_G = \De_{G_1}\cdots \De_{G_t}$. \end{prop}

\begin{proof} After suitable relabeling, the Laplacian matrix for $G$ is a  block diagonal matrix with diagonal blocks equal to the 
Laplacian matrices for $G_1, \ldots, G_t$. The result follows immediately. Alternatively, it can be deduced from Theorem \ref{poly}.
\end{proof}

\begin{prop} \label{zeropoly} Let $G$ a $d$-periodic graph. If $G$ contains a finite component, then its Laplacian polynomial $\De_G$ is identically zero. The converse statement is true if $G$ is unsigned.
\end{prop} 

\begin{proof} If $G$ contains a finite component, then some component orbit $G_i$ consists of finite components.  We have $\De_{G_i}=0$ by Theorem \ref{poly}, since all cycles of $\overline{G_i}$ represent trivial homology classes and hence have vanishing connection. By Proposition \ref{components}, $\De_G$ is identically zero. 

Conversely, assume $G$ is unsigned and every component is infinite. Each component of $\o G$ must contain a nontrivial cycle. We can extend this collection of cycles to a cycle rooted spanning forest $F$ with no additional cycles.  The corresponding summand  in Theorem \ref{poly} has positive constant coefficient. Since every summand has nonnegative constant coefficient,  $\De_G$ is not identically zero.

\end{proof}


\section{Plane 1-periodic graphs and links in solid tori} \label{medialsect} 
 
When a plane graph $G$ is 1- or 2-periodic, the medial construction in Section \ref{Links} produces a diagram $D$ of an infinite link $\ell$. It has a finite quotient diagram $\o D$ modulo the $\Z$- or $\Z^2$-action induced by the action on $G$. 

For $d=1$ we regard $\o D$ in an annulus $\A$. It describes a link $\o \ell = \o \ell_1 \cup \cdots  \cup \o \ell_\mu$ in a solid unknotted torus $V$. The complement $({\rm int}\ V) \setminus \o \ell$ is homeomorphic to $\S^3 \setminus \hat \ell$, where  $\hat \ell = \o \ell \cup C$ is the link formed by the union of $\o \ell$ with a meridian $C$ of $V$. The meridian acquires an orientation induced by the infinite cyclic action on $D$.  It is easy to see that every link with an unknotted component that has even linking number with the rest of the link arises in this way.

The following result relates the Laplacian polynomial $\De_G$ to the Alexander polynomial $\Delta_{\hat \ell}$.

\begin {theorem} \label{alex} Let $G$ be a plane 1-periodic graph and $\hat \ell$ the encircled link $\o \ell \cup C$. Then 
$$\De_G(x) \  {\buildrel \cdot \over =}\   (x-1)\ \Delta_{\hat \ell}(-1, \ldots, -1, x),$$
where ${\buildrel \cdot \over =}$  indicates equality up to multiplication by units in $\Z[x^{\pm 1}]$. 

\end{theorem} 

\begin{proof} The argument at the end of section \ref{col per} shows that the Laplacian group of a finite plane graph is isomorphic to the abelian core group of the associated link.
The same argument can be applied to any 1-periodic graph $G$ and its associated link $\ell$. (Either of the two unbounded regions of the link diagram can be used as the base region as we pass from from vertex colorings to Fox colorings via Dehn colorings.) 

Let $D$ be the diagram of $\ell$ obtained from $G$ by the medial construction. We regard $D$ as lying in a strip $I \times \R \subset \R^2$, with   $R= I \times [0, 1]$ a fundamental domain for the $\Z$-action. Then $D$ meets $R$ in a tangle diagram $D_0$. We label the arcs of $D_0$ meeting the ``top," $R \times \{1\}$, by $a_1, \ldots, a_n$ and those meeting the ``bottom," $R \times \{0\}$, by $a'_1, \ldots, a'_n$. (It can happen that some $a_i$ and $a'_j$ are identical.) Let $b, c, \ldots$ be labels for the remaining arcs of $D_0$. 

Define $B$ to be the quotient of the free abelian group on $a_1, \ldots, a_n, a'_1, \ldots, a'_n, b, c \ldots$ by the Fox relations (Figure \ref{fox}) of the crossings in $D_0$. Let $U$ be the free abelian on $u_1, \ldots, u_n$, and $f: U \to B$ (resp. $g: U \to B$)  the homomorphisms mapping each $u_i$ to $a_i$ (resp. $u_i$ to $a'_i$). 
The Laplacian module has the form 
\begin{equation}\label{module} \cdots \oplus_U B\oplus_U B \oplus_U \cdots \end{equation}
with identical amalgamations $B\ {\buildrel g \over \leftarrow}\ U\ {\buildrel f \over \rightarrow}\ B.$ The module action of $x$ merely shifts 
each summand $B$ one place to the right. 
Thus the Laplacian module ${\L}_G$ is the cokernel of the square matrix  $A$ with columns corresponding to the arcs of $D_0$ and rows recording the Fox relations as well as the relations $xa'_1 = a_1, \ldots, xa'_n = a_n$.

We claim that the matrix $A$  is an Alexander matrix of the link $\bar \ell \cup C$ with $x$ corresponding to a meridian $m$ of $C$ while the meridianal variables of $\ell$ are set equal to $-1$ (cf. \cite{BZ03}). To see this, consider Figure \ref{encircle}. The Alexander matrix has columns corresponding to $a_1, \ldots, a_n, a'_1, \ldots, a'_n, b, c \ldots$ and the meridian $m$. There are $n$ rows corresponding to the crossing relations in $D_0$, and additional rows for relations $m + x a'_i = a_i - m$, where $i = 1, \ldots, n$. 
(The relations, which can be determined by Fox calculus or by considering the appropriate infinite cyclic cover, are unaffected by the directions of the arcs of $D_0$. The $n-1$ arcs in the back of $C$ correspond to generators defined by the relations that arise from $n-1$ of the $n$ crossings in the back; the last relation is redundant and can be ignored. Hence the extra generators and relations can be disregarded.) We delete the column corresponding to $m$ in order to obtain an Alexander matrix. The rows that we have added now correspond to the relations $xa'_i =a_i$. The result is the matrix $A$. The Alexander polynomial of $\bar \ell \cup C$ is the determinant of $A$ divided by $x-1$.

\begin{figure}[H]
\begin{center}
\includegraphics[height=1 in]{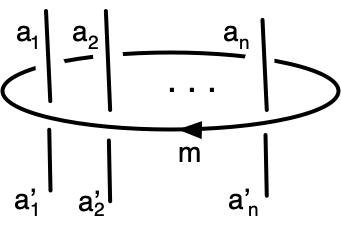}
\caption{Detail of $\bar \ell \cup C$}
\label{encircle}
\end{center}
\end{figure}

\end{proof}

\begin{prop} \label{nonzero} Let $G$ be a plane 1-periodic graph,  $D$ its associated link diagram, and $\ell$ the associated link. 
The Laplacian polynomial $\De_G$ is nonzero if and only if 
the Laplacian module ${\L}_G$ is a torsion module.
\end{prop}

\begin{proof}  The proposition follows from basic facts of commutative algebra. Since ${\L}_G$ has a square matrix presentation, $\De_G$ generates annihilator ideal of ${\L}_G$.  \end{proof}

\begin{prop} \label{equivalent} Let $G$ be a plane 1-periodic graph,  $D$ its associated link diagram, and $\ell$ the associated link. The following are equivalent.
\begin{enumerate}
\item The link $\ell$ has closed components. 
\item The group of 2-colorings of $D$ is infinite.
\item The Laplacian polynomial ${\De}_G$ reduced modulo 2 is identically zero.
\end{enumerate} 
\end{prop}

\begin{proof} 
Any 2-coloring of $D$ assigns a single color to every arc corresponding to a component of $\ell$. If the link has no closed components, then it has only finitely many components, and conversely. The equivalence of the first two statements follows.

Regard ${\L}_G$  as an abelian group. The vector space of 2-colorings of $D$ is isomorphic to ${\L}_G \otimes \Z/2$. Its dimension is the degree of the mod 2 reduction of $\De_G$, provided the reduced polynomial is nonzero; otherwise the dimension is infinite.
However, the dimension is also equal to the number of components of $\ell$. Hence the first and third statements are equivalent. \end{proof}  

The next proposition characterizes the leading coefficient of the Laplacian polynomial of a plane 1-periodic graph in terms of $\T$-colorings of its associated link. We will use  the notation in the proof of Theorem \ref{alex}.  Note that $\De_G(x)=\De_G(x^{-1})$, by Theorem \ref{poly}.

\begin{prop} \label{coef} Let $G$ be a plane 1-periodic graph,  $D$ its associated link diagram, and $D_0 \subset R$ a tangle diagram representing a fundamental region of $D$. Suppose $\De_G$ is nonzero. If we assign $0 \in \T$ to the arcs at the top of $D_0$, then the number of extensions to $\T$-colorings of $D_0$ is equal to the absolute value of the leading coefficient of $\De_G$. 
 \end{prop}

\begin{proof}  The $\T$-colorings of $D$ that assign $0$ to arcs labeled $a_1, \ldots, a_n$ are the elements of the dual group of the quotient $B/f(U)$. The group $B/f(U)$ is the cokernel of the specialized Alexander matrix $A$ constructed in the proof of Theorem \ref{alex} with variable $x$ set equal to $0$. The determinant of $A$ is the order of $B/f(U)$ as well as the absolute values of both the trailing and leading coefficients of $\De_G$. 
\end{proof} 

\begin{prop} \label{distinct} Let $G$ be a plane 1-periodic graph, $D$ its associated link diagram, and $D_0 \subset R$ a tangle diagram representing a fundamental region of $D$. Suppose that the coefficients of $\De_G$ are coprime. Then any two distinct $\T$-colorings of $D_0$ with the same color assignments of the top arcs have different color assignments of the bottom arcs. 
\end{prop}

\begin{proof} We prove the proposition by contradiction. Assume that there exist two $\T$-colorings of $D_0$ with the same color assignments to the top arcs and also the same assignments to the bottom. We subtract to get a nontrivial $\T$-coloring with all arcs on both top and bottom colored trivially. By duality, the group quotient $\o B = B/(f(U)+g(U))$ must be nonzero. Consider the quotient module $\o {\L}_G$ of ${\L}_G$ described by (\ref{module}) with all elements of $x^i f(U)$ and $x^ig(U)$ set equal to $0$; it is a direct sum of countably many copies of $\o B$ with the module action of $x$ given by translation. An integral matrix presenting $\o B$ as an abelian group also presents $\o {\L}_G$  as module over ${\cal R}_1
 = \Z[x, x^{-1}]$. The group must be torsion since ${\L}_G$ is, and its 0th-characteristic polynomial $\o \De_G$ must divide $\De_G$.  Since $\o \De_G$ is a nonzero constant and the coefficients of $\De_G$ are coprime,  $\o \De_G = \pm 1$.  Hence $\o B$ is a trivial group, a contradiction. \end{proof} 

The following corollary follows form Propositions \ref{coef} and \ref{distinct}.

\begin{cor} Assume that $G$ is a plane 1-periodic graph, $D$ its associated link diagram, and $D_0 \subset R$ a tangle diagram representing a fundamental region of $D$. Assume that $\De_G$ is monic. Given any color assignment to the top arcs of $D_0$ that extends to a $\T$-coloring of $D_0$, the colors of the bottom arcs are uniquely determined. 
\end{cor} 

\begin{example}  Consider the 1-periodic graph $G$ and its associated tangle diagram $D_0$ in Figure \ref{tangle}. 
It is easy to see that $\De_G(x) = -3x+6-3x^{-1}$, which has leading coefficient $-3$.  
The group $B$ of the associated tangle has generators $a_0, a_1, b_0, b_1, c$ and relations $2b_0 = a_0 +c,
2c = b_0 + a_1, 2 a_1 = c + b_1$. 

\begin{figure}[H]
\begin{center}
\includegraphics[height=2 in]{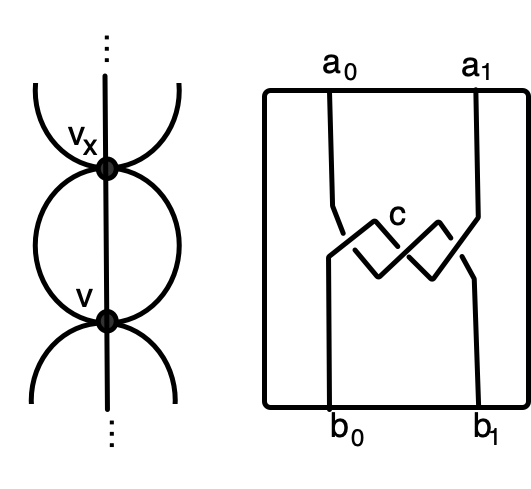}
\caption{1-peridodic graph $G$ and associated tangle diagram $D_0$}
\label{tangle}
\end{center}
\end{figure}

Since the subgroup $f(U)$ is generated by $a_0, a_1$, the quotient $B/f(U)$ is generated by $b_0, b_1, c$, with relations $2 b_0 = c, 2c = b_0, 0 = c + b_1$. If we choose 
a color $\g \in \T$ for the arc of $D_0$ labeled $c$, then the bottom arcs, labeled $b_0, b_1$,  must receive colors $2 \g, -\g$, respectively. Moreover, the assignment is a $\T$-coloring of $D_0$ provided that $3 \g = 0$. Hence $\g = 0, 1/3, 2/3$ (mod 1),
and there are exactly three $\T$-colorings of $D_0$ with the top arcs colored $0$, as expected from Proposition \ref{coef}. The three $\T$-colorings appear in Figure \ref{ctangle}.

\begin{figure}[H]
\begin{center}
\includegraphics[height=1.5 in]{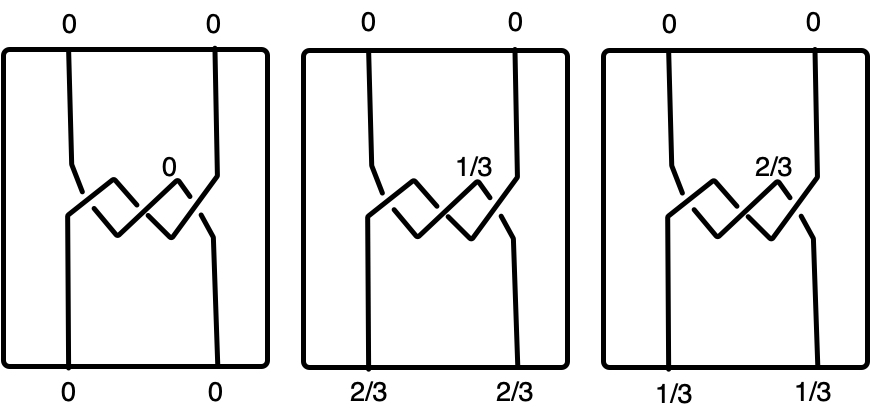}
\caption{$\T$-colorings with top arcs colored trivially}
\label{ctangle}
\end{center}
\end{figure}

\end{example}

\begin{example}\label{braid}
 The closure $\o \ell$ of any $2n$-braid arises from graph $\o G$ embedded in the annulus $\A$, since we can checkerboard shade a diagram of $\o \ell$ in $\A$ so that the border regions are unshaded.

The graph $\o G$ lifts to a 1-periodic graph $G$ in the plane. Consider its Laplacian polynomial $\De_G$.
We recall that the Burau representation associates to each generator $\si_i$ of the $2n$-braid group a block diagonal matrix 
$$I_{i-1} \oplus \begin{pmatrix} 1-t & t \\ 1 & 0 \end{pmatrix} \oplus I_{2n -i-1},$$ where $I_k$ denotes the $k \times k$ identity matrix.
Setting $t=-1$ produces presentation matrix for the group $B$ in the proof of Theorem \ref{alex}. The Laplacian polynomial 
of $G$ is the characteristic polynomial of the Burau matrix of the braid with $t=-1$. 

\end{example}

\section{Complexity growth of periodic graphs} \label{knots} 
When $G$ is a $d$-periodic graph with quotient $\o G$, we can consider the intermediate covering graphs  $G_\La$ in $\mathbb R^d/\La$, where $\La \subset \Z^d$ is a subgroup of $\La$ having finite index.
In this section we see that the growth of the torsion complexity $\k_{G_\La}$ as the index of $\La$ goes to infinity is determined by the Mahler measure of the Laplacian polynomial $\De_G$. 

We begin by reviewing the Mahler measure of polynomials.

\begin{definition} \label{mahler} The \emph{Mahler measure} of a nonzero polynomial 
$f(x_1, \ldots, x_d) \in \Rd$ is 
\begin{equation*} M(f) =\exp \int_0^1 \ldots \int_0^1 \log|f(e^{2\pi i \theta_1}, \ldots, e^{2\pi i \theta_d})| d\theta_1 \cdots d\theta_d. \end{equation*}

\end{definition} 

\begin{remark} (1)  The integral in Definition \ref{mahler} can be singular, but nevertheless it converges. 
(See \cite{EW99} for two different proofs.)  If $u_1, \ldots, u_d$ is another basis for $\Z^d$, then $f(u_1, \ldots, u_d)$ has the same logarithmic Mahler measure as $f(x_1, \ldots, x_d)$.  \smallskip

(2) If $f, g \in \Rd$, then $M(fg) = M(f)M(g)$. Moreover, $M(f) =1$ if and only if $f$ is a unit or a unit times a product of 1-variable cyclotomic polynomials, each evaluated at a monomial of $\Rd$
(see \cite{Sc95}). 

(3) When $d=1$, Jensen's formula shows that $M(f)$ can be described in a simple way. If 
$f(x) = c_s x^s+ \cdots c_1 x + c_0$, $c_0c_s \ne 0$,
$c$ is the leading coefficient of $f$, then
\begin{equation*} M(f) = |c_s| \prod_{i=1}^s \max \{ |\lambda_i|, 1\},\end{equation*}
where $\lambda_1, \ldots, \lambda_s$ are the roots of $f$. \smallskip

\end{remark}

\begin{theorem} \label{limit} If $G$ is a signed $d$-periodic graph with nonzero Laplacian polynomial $\De_G$,  then  
\begin{equation}  \limsup_{\langle \La \rangle \to \infty} \frac{1}{|\Z^d/\La|} \log \k_{G_\La}=
\log M(\De_G), \end{equation}
where $\La$ ranges over all finite-index subgroups of $\Z^d$, and $\langle \La \rangle$  denotes the minimum length of a nonzero vector in $\La$. When $d=1$,  the limit superior can be replaced by an ordinary limit. 
\end{theorem}

We call this limit the {\it complexity growth rate} of $G$, and denote it by $\g_G$.  Its relationship to the {\it thermodynamic limit} or {\it bulk limit} defined for a wide class of unsigned lattice graphs is discussed in  \cite{LSW14}, and also below in Section \ref{unsigned}. 

\begin{remark} \label{remarks} (1) The condition $\langle \La \rangle \to \infty$ ensures that fundamental region of $\La$ grows in all directions. 

(2) If $G$ is unsigned,  $\k_{G_\La}= \t_{G_\La}$ for every $\La$. In this case, Theorem \ref{limit} is proven in \cite{Ly05} with the limit superior replaced by ordinary limit. 

(3) When $d=1$, the finite-index subgroups $\La$ are simply $\Z/r\Z$, for $r >0$. In this case, we write $G_r$ instead of $G_\La$.

(4) When $d>1$,  a recent result of V. Dimitrov \cite{Di16} asserts that the limit superior in Theorem \ref{limit} is equal to the ordinary limit along sequences of sublattices $\La$ of the form $N\cdot \Z^d$, where $N$ is a positive integer. 

 \end{remark}

 \begin{figure}
\begin{center}
\includegraphics[height=3 in]{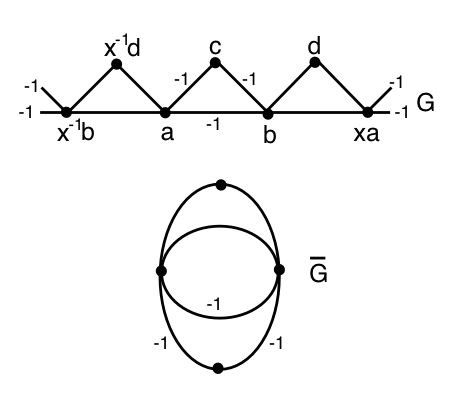}
\caption{1-Periodic graph $G$ with $\t_{G_r}=0$ for all $r \ge 1$}
\label{circulantdiagram}
\end{center}
\end{figure}

Before proving Theorem \ref{limit} we give an example that demonstrates the need for defining graph complexity as we do. 

\begin{example} \label{circ} Consider the 1-periodic graph $G$ in Figure \ref{circulantdiagram}. Generators for the Laplacian module are indicated. The Laplacian matrix is 
$$L_G= \begin{pmatrix} 0 &  1-x^{-1} & 1 & -x^{-1} \\ 1-x & 0 & 1 & -1 \\
1 & 1 & -2 & 0 \\ -x & -1 & 0 & 2 \end{pmatrix},$$ and 
$\De_G(x) = 9 (x-2+x^{-1}).$

The quotient  $G_2$ is the finite graph in Example \ref{milnor}. The  Laplacian matrix of any $G_r$ can be described as a block matrix obtained from $L_G$ by replacing $x$ by the companion (permutation) matrix for $x^r -1$, and any scalar $c$ by $c I_r$ (see \cite{SW00'}). It is conjugate to the diagonal block matrix ${\rm Diag}[L_G|_{x=1}, \ldots, L_G|_{x=\zeta^{r-1}}]$, where $\zeta$ is a primitive $r$th root of unity. The matrix $L_G|_{x=1}$ is the $4 \times 4$ Laplacian matrix of $\o G$, 
$$ L_{\o G}= \begin{pmatrix} 0 &  0 & 1 & -1\\ 0 & 0 & 1 & -1 \\
1 & 1 & -2 & 0 \\ -1 & -1 & 0 & 2 \end{pmatrix},$$ which has nullity 2. 
Hence the tree complexity $\t_{G_r}$ vanishes for every $r$. Nevertheless, by Theorem \ref{limit} the (torsion) complexity $\k_{G_r}$ is nontrivial and has exponential growth rate equal to $9$. One can verify directly that the Laplacian subgroup ${\cal L}_{G_r}$ is 
isomorphic to $\Z^2 \times (\Z/ 3^{r-1}\Z)^2.$\end{example} 

We proceed with the proof of Theorem \ref{limit}. 
 
 \begin{proof} The proof that we present is a direct application of a  theorem of D. Lind, K. Schmidt and T. Ward (see \cite{LSW90} or  Theorem 21.1 of \cite{Sc95}). We review the ideas for the reader's convenience. 
 
Recall that the Laplacian module ${\L}_G$ is the finitely generated module over the ring $\Rd$ with presentation matrix equal to the $n \times n$ Laplacian matrix $L_G$, and its Pontryagin dual group $\widehat {\L}_G$ is ${\rm Hom}({\L}_G, \T)$. The module actions of $x_1, \ldots, x_d$ determine commuting homeomorphisms $\s_1, \ldots, \s_d$ of $\widehat {\L}_G$. Explicitly, $(\s_j \rho)(a) = \rho(x_j a)$ for every $a \in {\L}_G$. Consequently, $\widehat {\G}_G$ has a $\Z^d$-action $\s: \Z^d \to \text{Aut}(\widehat {\L}_G)$.

The pair $(\widehat {\L}_G, \si)$ is an algebraic dynamical system, well defined up to topological conjugacy (that is, up to a homeomorphism of $\widehat {\L}_G$ respecting the $\Z^d$ action). In particular its periodic point structure is well defined.  

Topological entropy $h(\s)$ is a well-defined quantity associated to $(\widehat {\L}_G, \s)$, a measure of complexity of the $\Z^d$-action $\s$. We refer the reader to \cite{LSW90} or \cite{Sc95} for the definition.

For any subgroup $\La$ of $\Z^d$, a $\La$-periodic point is a member of $\widehat {\L}_G$ that is fixed by every element of $\La$. The set of $\La$-periodic points is a finitely generated abelian group isomorphic to the Pontryagin dual group ${\rm Hom}({\L}_G/\La {\L}_G), \T)$.

The group ${\L}_G/\La {\L}_G$ is the Laplacian module of the quotient graph $G_\La$. As a finitely generated abelian group, it decomposes 
as $\Z^{\b_\La} \oplus T({\L}_G/\La {\L}_G)$, where $\b_\La$ is the 
rank of ${\L}_G/\La {\L}_G$ and $T( \cdots )$ denotes the (finite) torsion subgroup. The Pontryagin dual group consists of $P_\La = |T({\L}_G/\La {\L}_G)|$ tori each of dimension $\b_\La$. By Theorem 21.1 of \cite{Sc95}, the topological entropy $h(\s)$ is:  \begin{equation*} h(\s) = \limsup_{\langle \La \rangle \to \infty}  \frac{1}{|\Z^d/\La|} \log P_\La =\limsup_{\langle \La \rangle \to \infty}  \frac{1}{|\Z^d/\La|} \log \k_\La.  \end{equation*}
Since the matrix $L_G$ that presents ${\L}_G$ is square, $h(\s)$ can be computed also as the logarithm of the Mahler measure $M(\det L_G)$ (see Example 18.7(1) of \cite{Sc95}).  The determinant of $L_G$ is, by definition, the Laplacian polynomial $\De_G$. Hence the proof is complete. 
\end{proof}

\section{Lehmer's question} In \cite{Le33} D.H. Lehmer 
asked the following question. 

\begin{question} Do there exist  integral polynomials with Mahler measures arbitrarily close but not equal to 1?  \end{question} 

Lehmer discovered the polynomial $x^{10} + x^9 - x^7 - x^6 - x^5 - x^4 - x^3 +x +1,$ which has Mahler measure equal to $1.17628...$.  Despite great effort including extensive computer-aided searches \cite{Bo80, Bo81, MS12, Mo98, Ra94}, no smaller value greater than 1 has been found, and Lehmer's question remains unanswered.

Topological and geometric perspectives of Lehmer's question have been found \cite{Hi03}.  In \cite{SW07} we showed that Lehmer's question is equivalent to a question about Alexander polynomials of fibered hyperbolic knots in the lens spaces $L(n, 1), n>0$. (Lens spaces arose from the need to consider polynomials $f(x)$ with $f(1) =n \ne 1$.) 
Here we present another, more elementary equivalence, in terms of graph complexity.

An integer polynomial $f(x)$ is \emph{reciprocal} if $x^{{\deg}f}f(x^{-1}) = f(x)$. 
We will say that a Laurent polynomial $f(x) \in {\cal R}_1$ is \emph{palindromic} if 
$f(x^{-1}) = f(x)$. Any reciprocal polynomial becomes palindromic after it is multiplied by $ x^j$ or $x^j(x+1)$, for suitable $j$. In \cite{Sm71} C. Smyth proved that any irreducible integral non-reciprocal polynomial other than $x$ or $x-1$ has Mahler measure at least as large as the real root of $x^3-x-1$ (approximately 1.324). Since Mahler measure is multiplicative, it suffices to restrict our attention to palindromic Laurent polynomials when investigating Lehmer's question.  

\begin{prop} \label{realized} A polynomial $\De(x)$ is the Laplacian polynomial of a 1-periodic graph if and only if it has the form $(x-2+x^{-1}) f(x)$, where $f(x)$ is a palindromic polynomial. 
\end{prop}

\begin{proof} The Laplacian polynomial $\De(x)$ of any $1$-periodic graph is palindromic.  This follows from the fact that the transpose of $L_G$ is $L_G$ with x replaced by $x^{-1}$.  Since the row-sums of $L_G$ become zero when we set $x=1$, $x-1$ divides $\De(x)$. (Both observations  follow also from Theorem \ref{poly}.) 
Palindromicity requires that the multiplicity of $x-1$ be even. Hence $\De(x)$ has the form $(x-2+x^{-1})f(x)$, where $f(x)$ is palindromic. 

In order to see the converse assertion, consider any polynomial of the form $p(x)=(x-2+x^{-1})f(x)$, where $f(x)$ is palindromic. Then $p(x)$ is also palindromic. Clearly, we can write $p(x)$
as a constant plus a sum of terms $\pm(x^s-2+x^{-s})$; but the constant must be 0 since $p(1)=0$. Then $p(x)$ is the Laplacian polynomial of a 1-periodic graph, constructed as in the following example. 
\end{proof} 

\begin{example} \label{circulant} Multiplying Lehmer's polynomial $f(x)= x^{10} + x^9 - x^7 - x^6 - x^5 - x^4 - x^3 +x +1$ by the unit $x^{-5}$ and then by $x-2+x^{-1}$ yields
$x^6-x^5-x^4+x^2+x^{-2}-x^{-4}-x^{-5}+x^{-6},$
which in turn can be written as 
$$(x^2-2+x^{-2}) -(x^4-2+x^{-4})-(x^5-2+x^{-5})+ (x^6 -2 + x^{-6}).$$
This is the Laplacian polynomial of a 1-periodic graph $G$. The quotient graph $\o G$ is easily described. It has a single vertex, two edges with sign $+1$ and two with $-1$. The $(+1)$-signed edges wind twice and six times, respectively,  around the annulus in the direction corresponding to $x$. The $(-1)$-signed edges wind four and five times, respectively, in the opposite direction.
\end{example}

\begin{theorem} \label{LQ} Lehmer's question is equivalent to the following. 
Given $\e >0$, does there exist a 1-periodic graph $G$ such that 
$$1 < \lim_{r \to \infty}  (\t_{G_r})^{1/r} < 1+\e?$$ 
\end{theorem}

\begin{proof} When investigating Lehmer's question it suffices to consider polynomials of the form $(x-2+x^{-1})f(x)$, where $f(x)$ is  palindromic and irreducible. By Proposition \ref{realized} any such polynomial is realized as the Laplacian polynomial of a 1-periodic graph $G$ with a single vertex orbit. As in Example \ref{circ} the Laplacian matrix $L_{G_r}$ of any finite quotient $G_r$ can be obtained from $(x-2+x^{-1})f(x)$ by substituting for $x$ the companion matrix for $x^r-1$. Hence the nullity of $L_{G_r}$ is 1 provided that $f(x)$ is not a cyclotomic polynomial (multiplied by a unit), a condition that we can assume without loss of generality. 
Hence $\kappa_{G_r} = \tau_{G_r}$ for each $r$ (see discussion following Definition \ref{complexity}.) Theorem \ref{limit} completes the proof. 

\end{proof}

\begin{remark} (1) The closure of the 16-braid $(\si_1\si_2)^2(\si_1\si_2\si_3\si_4)^4(\si_1\si_2 \cdots \si_{15})^5$ is a link $\o \ell$ associated with a plane graph $\o G$ in the annulus (see Example \ref{braid}). The Laplacian polynomial $\De_G$ is 
$$x^{-7}(x-1)(x^2+1)(x^{10}-x^9-x^6+x^5-x^4-x+1).$$  Its Mahler measure is 1.35098$\ldots$. This is the smallest Mahler measure greater than 1 that we have yet found for any plane graph. 

(2) The conclusion of Theorem \ref{LQ} does not hold if we restrict ourselves to unsigned graphs. By Theorem \ref{absolute} below,  the Mahler measure of the Laplacian polynomial of any 1-periodic graph with all edge signs equal to 1 is at least 2. \bs

(3) If a 1-periodic graph $G$ as in Example \ref{circulant} can be found with $M(D_G)$ less than Lehmer's value $1.17628...$, then by results of  \cite{Mo08} some edge of $\o G$ must wind around the annulus at least 29 times. \end{remark}

The cyclic 5-fold cover of the graph in Example \ref{circulant} contains the complete graph on 5 vertices, and hence it is nonplanar.  If the answer to the following question is yes, then Lehmer's question is equivalent to a question about determinant density of links (see Remark \ref{rem}(3)).

\begin{question} \label{question1} Is Theorem \ref{LQ} still true if we require that the graphs $G$ be planar?  
\end{question}

We conclude this section with a result that will be used in the next section, but holds for signed as well as unsigned graphs.  
It concerns complexity growth of a $d$-periodic graph that is a union of disjoint $d'$-periodic graphs for some $d'<d$.

Suppose $H$ is a subgraph of  a $d$-periodic graph $G$ consisting of one or more connected components of $G$, such that the orbit of $H$ under $\Z^d$ is all of $G$.  Let 
$\Gamma < \Z^d$ be the stabilizer of $H$.  Then $\Gamma\cong \Z^{d'}$ for some $d' < d$, and its action on $H$ can be regarded as a cofinite free action of $\Z^{d'}$.   Consider the limit $$\g_H=\lim_{\<\La\>\to\infty} \frac{1}{|\Gamma/\La|} \log \k_{H_\La}$$
where $\La$ ranges over finite-index subgroups of $\Gamma$.

\begin{lemma}\label{new} Under the above conditions we have $\g_G=\g_H$.
\end{lemma}

\begin{proof}
Let $\La$ be any finite-index subgroup of $\Z^d$.  Then $H$ is invariant under $\La\cap\G$. The image of $H$ in the quotient graph $G_\La$ is isomorphic to  $H_{\La\cap\G}$.  

Note that the quotient $\overline H$ of $H$ by the action of $\G$ is isomorphic to $\overline G$, since the $\Z^d$ orbit of $H$ is all of $G$.  Since $G_\La$ is a $|\Z^d /\La|$-fold cover of $\overline G$ and $H_{\La\cap\G}$ is a $|\G/(\La\cap\G)|$-fold cover of $\overline H$, $G_\La$ comprises $k=|\Z^d /\La| / |\G/(\La\cap\G)|$ mutually disjoint translates of a graph that is isomorphic to $H_{\La\cap\G}$. Hence $\k_{G_\La}=\k_{H_{\La\cap\G}}^k$ and
$$\frac{1}{|\Z^d /\La|} \log \k_{G_\La}= \frac{1} {|\G/(\La\cap\G)|} \log \k_{H_{\La\cap\G}}.$$
Since $\<\La\cap\G\>\to\infty$ as $\<\La\>\to\infty$, we have $\g_G=\g_H$.
\end{proof}

\section{Complexity growth of unsigned periodic graphs} \label{unsigned}

It is natural to ask whether the Mahler measure of Laplacian polynomials of signed graphs differs in appreciable ways from unsigned graphs. Proposition \ref{absolute} answers emphatically yes. 

Throughout the section $G$ denotes an unsigned $d$-periodic graph. For this case, the complexity growth rate $\g_G$ is also the growth rate of the number of spanning trees of finite quotients $G_\La$.   Thus contracting or deleting an edge orbit of $G$ will not increase $\g_G$.


Denote by $R= R(\La)$ a fundamental domain of $\La$. Let $G\vert_R$ be the full unsigned subgraph of $G$ on vertices $v_{i, \n},\ \n \in R$. We denote by $\ell_R$ the corresponding medial link. 

 If  $G\vert_R$ is connected for each $R$, then $\{\tau_{G_\La}\}$ and $\{\tau_{G\vert_R}\}$ have the same exponential growth rates. (See Theorem 7.10 of \cite{LSW14} for a short, elementary proof. A more general result is Corollary 3.8 of \cite{Ly05}.) The \emph{bulk limit} is defined by $\g_G/|V(\o G)|$. 

\begin{example} The \emph{$d$-dimensional grid graph} $\Gr_d$ is the unsigned graph with  vertex set $\Z^d$ and single edges 
connecting each pair of vertices of distance 1.  Its Laplacian polynomial is 
\begin{equation*} \De(\Gr_d) = 2d -x_1-x_1^{-1}- \cdots - x_d-x_d^{-1}. \end{equation*}
When $d=2$, it is a plane graph. The graphs links $\ell_R$ are indicated in Figure \ref{gridlinks} for $\La = \langle x_1^2, x_2^2 \rangle$ on left and $\La= \langle x_1^3, x_2^3 \rangle$ on right.

\begin{figure}
\begin{center}
\includegraphics[height=2 in]{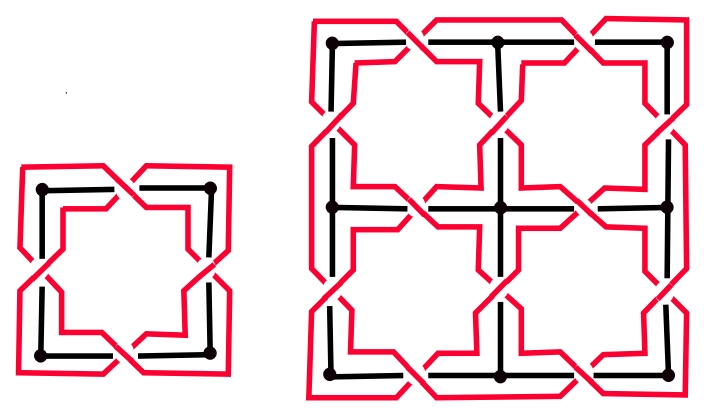}
\caption{Graphs $(\Gr_2)_R$ and associated links, $\La = \langle x_1^2, x_2^2 \rangle$ and $\langle x_1^3, x_2^3 \rangle$}
\label{gridlinks}
\end{center}
\end{figure}

\end{example}

The {\it determinant} of a link $\ell$, denoted here by ${\rm det}(\ell)$, is the absolute value of its 1-variable Alexander polynomial evaluated at $-1$.   
It follows from the Mayberry-Mott theorem \cite{BM54} that if $\ell$ is an alternating link that arises by the medial construction from a finite plane graph, edge signs $\pm 1$ allowed, then ${\rm det}(\ell)$ is equal to the tree complexity of the graph (see Appendix A.4 in \cite{BZ03}).
The following corollary is an immediate consequence of Theorem \ref{limit}.  It has been proven independently by Champanerkar and Kofman \cite{CK16}.

\begin{cor}\label{detgrowth} Let $G$ be a connected $d$-periodic unsigned plane graph, $d =1$ or $2$. Then $$\lim_{\langle \La \rangle \to \infty}  \frac{1}{\vert \Z^d/\La\vert} \log {\rm det}(\ell_R) = 
\g_{\De_G}.$$ \end{cor}

\begin{remark} \label{rem} (1) We regard the limit in that statement of Corollary \ref{detgrowth} as a \textit{determinant density} of the collection of links $\{\ell_R\}$. There are other ways to define it (e.g., dividing by the number of crossings of the diagram for $\ell_R$).

(2) In \cite{CKP15} the authors consider as well more general sequences of links.  When $G = \Gr_2$, their results imply that:

$$\lim_{\langle \La \rangle \to \infty} \frac{2 \pi}{c(\ell_R)} \log {\rm det}(\ell_R)=  v_{oct},$$
where $c(\ell_R)$ is the number of crossings of $\ell_R$ and $v_{oct} \approx 3.66386$ is the volume of the regular ideal octohedron. 

(3) If Question \ref{question1} has an affirmative answer then Lehmer's question becomes a question about link determinants. 

\end{remark}

Grid graphs are the simplest unsigned $d$-periodic graphs,  as the 
following theorem shows. 

\begin{theorem} \label{min} If $G$ is an unsigned connected $d$-periodic graph, then $\g_G\ge \g_{\Gr_d}$. 
\end{theorem}

Asymptotic results about the Mahler measure of certain families of polynomials have been obtained elsewhere. However, the graph theoretic methods that we employ to prove Theorem \ref{min} are different from techniques used previously. 

\begin{proof}
Consider the case in which $G$ has a single vertex orbit. Then for some $u_1,\ldots, u_m\in \Z^d$, with $m \ge d$, the edge set $E(G)$ consists of edges from $v$  to $u_i \cdot v$ for each $v\in V$ and $i=1,\ldots,m$.  Since $G$ is connected, we can assume after relabeling that $u_1, \ldots, u_d$ generate a finite-index subgroup of $\Z^d$. Let $G'$ be the $\Z^d$-invariant subgraph of $G$ with edges from $v$  to $u_i \cdot v$ for each $v\in V$ and $i=1,\ldots,d$.  Then $G'$ is the orbit of a subgraph of $G$ that is isomorphic to $\Gr_d$, and so by Lemma \ref{new}, $\g(\Gr_d)=\g(G')\le\g(G)$.

We now consider a connected graph $G$ having vertex families $v_{1, \n}, \ldots, v_{n, \n}$, where $n >1$. Since $G$ is connected, there exists an edge $e$ joining $v_{1, {\bf 0}}$ to some $v_{i, \n}$. Contract the edge orbit $\Z^d \cdot e$ to obtain a new graph $G'$ having cofinite free $\Z^d$-symmetry and complexity growth rate no greater than that of $G$.
Repeat the procedure with the remaining vertex families so that only $v_{1, \n}$ remains. The proof in the previous case of a graph with a single vertex orbit now applies. 
\end{proof}

\begin{remark} The conclusion of Theorem \ref{min} does not hold without the hypothesis that $G$ is connected. Consider the 2-periodic graph $G$ obtained from $\Gr_2$ by removing all vertical edges, so that $G$ consists of 
countably many copies of $\Gr_1$ . Then $\g_G =\g_{\Gr_1}=0$ while $\g_{\Gr_2} > 0$. \end{remark}

The following lemma, needed for the proof Proposition \ref{absolute}, is of independent interest. 

\begin{lemma} \label{nonincreasing} The sequence of complexity growth rates $\g_{\De_{\Gr_d}}$ is nondecreasing. 
\end{lemma} 

\begin{proof}  Consider the grid graph $\Gr_d$. Deleting all edges 
in parallel to the $d$th coordinate axis yields a subgraph $G$ consisting of countably many mutually disjoint translates of $\Gr_{d-1}$. By Lemma \ref{new},   $\g_{\Gr_{d-1}}=\g_G  \le \g_{\Gr_d}$. 
\end{proof}

Doubling each edge of $\Gr_1$ results in a graph with Laplacian polynomial $2(x-2+x^{-1})$, which has Mahler measure $2 M (x-2+x^{-1}) = 2$. We show that this graph realizes the minimum nonzero complexity growth rate. 

\begin{prop} \label{absolute} (Complexity Growth Rate Gap) Let $G$ be any unsigned $d$-periodic graph. If $\g_G \ne 0$, then 
$$\g_G \ge \log 2.$$            
\end{prop}

Although $\De_{\Gr_d}$ is relatively simple, the task of computing its Mahler measure is not. It is well known and not difficult to see that $\g_{\Gr_d} \le \log 2d$. We will use a theorem of N. Alon \cite{Al90} to show that $\g_{\Gr_d}$ approaches $\log 2d$ asymptotically. 

\begin{theorem} \label{gridlim} (1) $\g_{\Gr_d} \le \log 2d$, for all $d \ge 1$.\\ 
(2) $\lim_{d \to \infty} \g_{\Gr_d}- \log 2d = 0.$ \end{theorem}

%


\noindent {\it Proof of Proposition \ref{absolute}.}  By Lemma \ref{new} it suffices to consider a connected $d$-periodic graph $G$ with $\g_G$ nonzero. Note that $\g_{\Gr_1} =0$ while $\g_{\Gr_2} \approx 1.165$ is greater than $\log 2$.  By Theorem \ref{min} and Lemma \ref{nonincreasing} we can  assume that $d=1$. 

If $G$ has an orbit of parallel edges, we see easily that $\g_G \ge \log 2$.  Otherwise, we proceed as in the proof of  Theorem \ref{min}, contracting edge orbits to reduce the number of vertex orbits  without increasing the complexity growth rate.  If at any step we obtain an orbit of parallel edges, we are done; otherwise we will obtain a graph $G'$ with a single vertex orbit and no loops. 
If $G'$ is isomorphic to $\Gr_1$ then $G$ must be a tree; but then $\g_G=0$, contrary to our hypothesis. So $G'$ must have at least two edge orbits.  
Deleting excess edges, we may suppose $G'$ has exactly two edge orbits. 

The Laplacian polynomial $\De_{G'}$ has the form $4-x^r-x^{-r}- x^s-x^{-s}$, for some positive integers $r, s$. Reordering the vertex set of $G'$, we can assume without loss of generality that $r=1$. The following calculation is based on an idea suggested to us by Matilde Lalin.

$$\log M(\De_{G'}) = \int_0^1 \log \vert 4- 2 \cos( 2 \pi \theta)-2 \cos(2 \pi s \theta) \vert \ d\theta$$

$$=\int_0^1 \log \vert  2(1-\cos  (2 \pi \theta)) + 2(1 - \cos(2 \pi s \theta)) \vert \ d\theta$$

$$=\int_0^1 \log  \bigg( 4\sin^2(\pi \theta) + 4\sin^2(\pi s \theta) \bigg) \ d\theta.$$

\ni Using the inequality $(u^2+v^2) \ge 2 u v$, for any nonnegative $u, v$, we have:

$$\log M(\De_{G'}) \ge \int_0^1 \log \bigg( 8 \vert \sin( \pi \theta)\vert\  \vert \sin( \pi s \theta) \vert \bigg)\ d\theta$$

$$= \log 8 + \int_0^1 \log \vert \sin(  \pi \theta) \vert \ d\theta + \int_0^1 \log \vert \sin( \pi s \theta) \vert \ d\theta$$

$$= \log 8 + \int_0^1 \log \sqrt{\frac{1-\cos(2 \pi \theta)}{2}}\ d \theta + \int_0^1 \log \sqrt{\frac{1-\cos(2  \pi s\theta)}{2}}\ d \theta$$

$$=\log 8 + \int_0^1 \frac{1}{2} \log \bigg(\frac{2 - 2 \cos(2 \pi \theta)}{4}\bigg)\ d \theta +   \int_0^1 \frac{1}{2} \log \bigg(\frac{2 - 2 \cos(2 \pi s \theta)}{4}\bigg)\ d \theta$$

$$= \log 8 + \frac{1}{2}m(2 - x -x^{-1}) -\frac{1}{2} \log 4 + \frac{1}{2}m(2 - x^s - x^{-s})- \frac{1}{2} \log 4$$

$$= 3\log 2 + 0 - \log 2 +0 -\log 2 = \log 2.$$
\qed\\

Our proof of Theorem \ref{gridlim} depends on the following result of Alon. 

\begin{theorem}\label{lowerbound}\cite{Al90}   If $G$ is a finite connected $\rho$-regular unsigned graph, then 
$$\t_G \ge [\rho(1-\epsilon(\rho))]^{|V(G)|},$$ 
where $\epsilon(\rho)$ is a nonnegative function with $\epsilon(\rho)\to 0$ as $\rho\to\infty$.
\end{theorem}

\noindent{\it Proof of Theorem \ref{gridlim}.} (1) The integral representing the logarithm of the Mahler measure of $\De_{\Gr_d}$ can be written
$$\int_0^1 \cdots \int_0^1 \log \bigg\vert2d - \sum_{i=1}^d 2 \cos(2 \pi \theta_i)\bigg\vert d\theta_1\cdots d\theta_d$$
$$= \log 2d + \int_0^1 \cdots \int_0^1 \log \bigg\vert1+\sum_{i=1}^d \frac{ \cos( 2 \pi \theta_i)}{d}\bigg\vert d\theta_1\cdots d\theta_d$$
$$= \log 2d + \int_0^1 \cdots \int_0^1 -\sum_{k=1}^\infty \frac{(-1)^k}{k}\bigg( \frac{\sum_{i=1}^d \cos(2 \pi \theta_i)}{d} \bigg)^k d\theta_1\cdots d\theta_d.$$
By symmetry, odd powers of $k$ in the summation contribute zero to the integration. Hence 
$$\log(\De_{\Gr_d})= \log 2d - \int_0^1 \cdots \int_0^1 \sum_{k=1}^\infty \frac{1}{2k}\bigg( \frac{\sum_{i=1}^d \cos(2 \pi \theta_i)}{d} \bigg)^{2k} d\theta_1\cdots d\theta_d \le \log 2d.$$

(2) Let $\La$ be a finite-index subgroup of $\Z^d$. Consider the quotient graph $(\Gr_d)_\La$. The cardinality of its vertex set is $|\Z^d/\La|$. The main result of \cite{Al90}, cited above as Theorem \ref{lowerbound}, implies that 
$$\t_{(\Gr_d)_\La} = \bigg((2 d) (1 - \mu(d))\bigg)^{|\Z^d/\La|},$$
where $\mu$ is a nonnegative function such that $\lim_{d\to \infty} \mu(d) =0$.  
Hence 
$$\lim_{d \to \infty}  \bigg (\frac{1}{|\Z^d/\La|} \log \t_{(\Gr_d)_\La} - \log 2 d\bigg ) = \lim_{d\to \infty} \log (1-\mu(d)) =0.$$
Theorem \ref{limit} completes the proof.  
\qed \\

\begin{remark} One can evaluate $\log M(\De(\Gr_d))$ numerically and obtain an infinite series representing
$\g_{\Gr_d} - \log 2d$. However, showing rigorously that the sum of the series approaches zero as $d$ goes 
to infinity appears to be difficult. (See \cite{SW00}, p.~16 for a heuristic argument.) 
\end{remark}

\bigskip

\ni Department of Mathematics and Statistics,\\
\ni University of South Alabama\\ Mobile, AL 36688 USA\\
\ni Email: \\
\ni  silver@southalabama.edu\\
\ni swilliam@southalabama.edu

\end{document}